\documentclass[12pt, twoside]{article}
\usepackage{amsmath,amssymb}
\usepackage{amsthm}
\usepackage{graphicx}
\usepackage{epsfig}
\usepackage{cite}
\usepackage{subfigure}
\usepackage{hyperref}

\newtheorem{theorem}{Theorem}[section]
\newtheorem{lemma}{Lemma}[section]

\newtheorem{assumption}{Assumption}[section]

\numberwithin{equation}{section}

\numberwithin{equation}{section} \makeatletter
\begin{document}

\title{\textbf{Hopf bifurcation of the Michaelis-Menten \\type ratio-dependent predator-prey \\model with age structure}}
\author{Xiangming Zhang{\small \textsc{$^{a,}$\thanks{%
Research was partially supported by NSFC (Grant No. 11471044 and 11771044) and the
Fundamental Research Funds for the Central Universities.}}} and Zhihua Liu%
{\small \textsc{$^{a,*,}$}}\textsc{\thanks{%
{\small Corresponding author. \newline \indent~~E-mail addresses: xiangmingzhang@mail.bnu.edu.cn (X. Zhang), zhihualiu@bnu.edu.cn (Z. Liu).}}} \\
$^{a}$School of Mathematical Sciences, Beijing Normal University,\\
Beijing, 100875, People's Republic of China}

\date{}
\maketitle
\begin{abstract}
This paper is devoted to the study of
a predator-prey model with predator-age structure that involves Michaelis-Menten type ratio-dependent functional response. We study some dynamical properties of the model by using the theory of integrated semigroup and the Hopf bifurcation theory for semilinear equations with non-dense domain. The existence of Hopf bifurcation is established by regarding the biological maturation period $\tau$ as the bifurcation parameter. The computer simulations and sensitivity analysis on parameters are also performed to illustrate the conclusions.


\end{abstract}

\section{Introduction}

\noindent

In the predator-prey population dynamics, one of the most fashionable and considerable mathematical model sketching a predator-prey interaction is the following well-known Lotka-Volterra type predator-prey model with Michaelis-Menten (or Holling type II) functional response \cite{FreedmanHerbertI-1980}:
\begin{equation}\label{predator-prey}
  \left\{
     \begin{array}{l}
       x'(t)=rx(1-\frac{x}{K})-p(x)y, \\
       y'(t)=\eta p(x)y-\sigma y, \\
       x(0)>0, y(0)>0, \\
     \end{array}
   \right.
\end{equation}
where $x$ and $y$ denote prey and predator density, respectively; $r$, $K$, $\eta$ and $\sigma$ are positive constants that denote prey intrinsic growth rate, carrying capacity of prey of the environment that is frequently determined by the available sustaining resources, coefficient for the conversion that predator intake to per capital prey and predator mortality rate. $p(x)=\frac{\alpha x}{m+x}$ is the
Michaelis-Menten (or Holling type II) functional response, where $\alpha>0$ is the capturing rate and $m>0$ is the half saturation constant.
From a biological point of view,
the so-called predation term $p(x)$, which is the functional response of the predator to the change in the density of prey, generally demonstrates some saturation effect. Obviously the function $p(x)=\frac{\alpha x}{m+x}$ depends merely on prey density $x$.
Therefore it is often called a \emph{prey-dependent} response function. The model (\ref{predator-prey}) demonstrates the well-known ``paradox of enrichment" \cite{HairstonSmithSlobodkin-JTB-1960,Rosenzweig-Science-1971} and the so-called ``biological control paradox" \cite{Luck-TEE-1990}.
According to \cite{Akcakaya-Ecology-1995,CosnerDeAngelisAultOlson-TPB-1999}, a more suitable realistic predation term of predator-prey model depends upon the amount of prey that each predator can share.
This conclusion is supported by numerous fields, laboratory experiments and observations \cite{ArditiGinzburg-JTB-1989}.
On the basis of the Michaelis-Menten (or Holling type II function), \cite{ArditiGinzburg-JTB-1989} proposed the following response function of the form
\begin{equation*}
  p\left(\frac{x}{y}\right)=\frac{\alpha\frac{x}{y}}{m+\frac{x}{y}}=\frac{\alpha x}{my+x},
\end{equation*}
where $x$ and $y$ stand for prey and predator density, respectively. Such a functional response is usually called a \emph{ratio-dependent} response function.

The difference between ratio-dependent models and prey-dependent models has been discussed in \cite{Berryman-E-1992}.
Comparing the prey-dependent predator-prey models, \cite{ArditiGinzburg-JTB-1989} graphically analysed  the advantages of the
ratio-dependent predator-prey systems by using the isocline method.
In this paper, we will contribute to the Hopf bifurcation analysis for ratio-dependent predator-prey with age structure rather than discuss the
general ecological significance of this class of models.

Combined with local stability analysis and simulations, \cite{ArditiGinzburg-JTB-1989, Berryman-E-1992} demonstrated that the ratio-dependent models have ability of producing more complex and more reasonable dynamics \cite{KuangBeretta-JMB-1998,JostArinoArditi-BMB-1999,XiaoRuan-JMB-2001}.
In document \cite{KuangBeretta-JMB-1998}, the authors discussed the model (\ref{predator-prey}) and considered the global behaviors of solutions of model (\ref{predator-prey}). They also demonstrated that ratio-dependent predator-prey systems are rich in boundary dynamics and if the positive steady state of the system (\ref{predator-prey}) is locally asymptotically stable, then
the system has no nontrivial positive periodic solutions. \cite{XiaoRuan-JMB-2001} studied the qualitative behavior of a class of ratio-dependent predator-prey system at the origin and shown that there can exist numerous kinds of topological structures in a neighborhood of the origin.

Age is one of the most prevalent and significant parameters structuring a population. In a word, many internal variables, at the level of the single individual, are inevitably depending upon the age because different age implies different reproduction and survival capacities, and, also different behaviors. Recently the papers about age structure become increasingly commonplace (see \cite{Khajanchi-AMC-2014,MimmoIannelli-1995,WangLangZou-NARWA-2017,XuZhang-DCDS-B-2016,
YangLiZhang-IJB-2016,YangRuanXiao-MBE-2015,LiuLi-JNS-2015,TangLiu-AMM-2016,WangLiu-JMAA-2012,CushingSaleem-JMB-1982}). However, most of the results on age structure model focus on the existence, bounded and stability of the positive solutions \cite{Khajanchi-AMC-2014,MimmoIannelli-1995,WangLangZou-NARWA-2017,XuZhang-DCDS-B-2016,
YangLiZhang-IJB-2016,YangRuanXiao-MBE-2015,CushingSaleem-JMB-1982}. \cite{TangLiu-AMM-2016} investigated the Hopf bifurcation of prey-dependent predator-prey model with predator age structure. The authors formulated the model as an abstract non-densely defined Cauchy problem and derived the existence of Hopf bifurcation. However, they considered the predation term with prey-dependent response function.

Motivated by the references \cite{KuangBeretta-JMB-1998,XiaoRuan-JMB-2001,TangLiu-AMM-2016},
we reconsider the Michaelis-Menten predator-prey model (\ref{predator-prey}) with an predator-age structure.
As far as we know, the age structure model can be considered as an abstract Cauchy problem with non-dense domain.
In this paper, we attempt to investigate the model (\ref{system}) by means of the theory of integrated semigroup and the Hopf bifurcation theory \cite{LiuMagalRuan-ZAMP-2011}.
Furthermore, the existence of Hopf bifurcation is investigated and the numerical simulations are also presented
to support our conclusions. Our results show that when the bifurcation parameter $\tau$ passes through a critical value, the Hopf bifurcation occurs.

The rest of this paper is organized as follows. In Section 2, we first describe the Michaelis-Menten type ratio-dependent
predator-prey model with age structure. Then this model is reformulated as an abstract non-densely defined Cauchy problem and the equilibria, linearized equation and characteristic equation are investigated. In Section 3, we show the existence of Hopf bifurcation. The numerical results are presented in Section 4. Sensitivity analysis are carried out in Section 5. Some conclusions are given in Section 6.

\section{Preliminaries}
\noindent
\subsection{Model description}

\noindent

In this section, we introduce the Michaelis-Menten type ratio-dependent
predator-prey model with age structure. Let $a$ be the predator-age variable. $u(t,a)$ is the distribution function of the predators over predator-age $a$ at time $t$. Then the number of the predators at time $t$ equals to $\int_{0}^{+\infty}{u(t,a)da}$. Correspondingly, the predation term that involves the ratio-dependent response function is given by
\begin{equation*}
  p\left(\frac{V(t)}{\int_{0}^{+\infty}{u(t,a)da}}\right)=\frac{\alpha\frac{V(t)}{\int_{0}^{+\infty}{u(t,a)da}}}{m+\frac{V(t)}{\int_{0}^{+\infty}{u(t,a)da}}}
  =\frac{\alpha V(t)}{m\int_{0}^{+\infty}{u(t,a)da}+V(t)}.
\end{equation*}

\noindent
In mathematical terms, the dynamics of such a system of predator and prey may be written as
\begin{equation}\label{system}
\left\{
\begin{array}{l}
\frac{\partial u(t,a)}{\partial t}+\frac{\partial u(t,a)}{\partial a} = -\sigma u(t,a),\\
\frac{dV(t)}{dt}=rV(t)\left(1-\frac{V(t)}{K}\right)-\frac{\alpha V(t)}{m\int_{0}^{+\infty}{u(t,a)da}+V(t)}\int_{0}^{+\infty}{u(t,a)da} ,\\
u(t,0)=\eta \frac{\alpha V(t)}{m\int_{0}^{+\infty}{u(t,a)da}+V(t)}\int_{0}^{+\infty}{\beta(a)u(t,a)da},\quad t>0,\\
u(0,\cdot)=u_{0}\in L_{+} ^{1}((0, + \infty ),\mathbb{R}),\quad V(0)=V_{0}\geq0,
\end{array}
\right.
\end{equation}
where $V$ is the prey density; $r=\Lambda-\mu$ is the intrinsic growth rate of the prey, and the other parameters are the same as the model (\ref{predator-prey}). Here and subsequently, $\beta(a)$ is an age-specific fertility function related to predator-age $a$ and satisfies the following assumption \ref{assumption1}.
\begin{assumption}\label{assumption1}
Assume that
\begin{equation*}
  \beta(a):=\left\{
               \begin{array}{cl}
                 \beta^{*}, &\quad \mbox{if} \quad a\geq \tau, \\
                 0, &\quad \mbox{if} \quad a\in (0,\tau), \\
               \end{array}
             \right.
\end{equation*}
where $\tau>0$ and $\beta^{*}>0$. Additionally, it is beneficial and reasonable to
assume that the predator population shows a stable trend. That is, $\int_{0}^{+\infty}{\beta(a)e^{-\sigma a}da}=1$, where $e^{-\sigma a}$ denotes the survival probability.
\end{assumption}

\noindent
\subsection{Rescaling time and age}

\noindent

In this subsection, our destination is to obtain a smooth dependency of the system (\ref{system}) with respect to $\tau$ (i.e., in order to use the parameter $\tau$ as a bifurcation parameter). We first normalize $\tau$ in (\ref{system}) by the time-scaling and age-scaling
\begin{equation*}
  \hat{a}=\frac{a}{\tau} \quad \mbox{and} \quad \hat{t}=\frac{t}{\tau},
\end{equation*}
and the following distribution
\begin{equation*}
  \hat{V}(\hat{t})=V(\tau\hat{t}) \quad \mbox{and} \quad \hat{u}(\hat{t},\hat{a})=\tau u(\tau\hat{t},\tau\hat{a}).
\end{equation*}
For abbreviation, after the change of variables we drop the hat notation and obtain the following new system
\begin{equation}\label{newsystem}
\left\{
\begin{array}{l}
\frac{\partial u(t,a)}{\partial t}+\frac{\partial u(t,a)}{\partial a} = -\tau\sigma u(t,a),\\
\frac{dV(t)}{dt}=\tau\left[rV(t)\left(1-\frac{V(t)}{K}\right)-\frac{\alpha V(t)}{m\int_{0}^{+\infty}{u(t,a)da}+V(t)}\int_{0}^{+\infty}{u(t,a)da}\right],\\
u(t,0)=\tau\left[\eta\frac{\alpha V(t)}{m\int_{0}^{+\infty}{u(t,a)da}+V(t)}\int_{0}^{+\infty}{\beta(a)u(t,a)da}\right],\quad t>0,\\
u(0,\cdot)=u_{0}\in L_{+} ^{1}((0, + \infty ),\mathbb{R}),\quad V(0)=V_{0}\geq0,
\end{array}
\right.
\end{equation}
where the new function $\beta(a)$ is given by
\begin{equation*}
  \beta(a)=\left\{
             \begin{array}{cc}
               \beta^{*}, & \quad \mbox{if} \quad a \geq 1,\\
               0, & \quad \mbox{otherwise}, \\
             \end{array}
           \right.
\end{equation*}
and
\begin{equation*}
  \int_{\tau}^{+\infty}{\beta^{*}e^{-\sigma a}da}=1\Leftrightarrow \beta^{*}=\sigma e^{\sigma \tau},
\end{equation*}
where $\tau\geq0$, $\beta^{*}>0$.
\noindent

With the notation $V(t):=\int_{0}^{+\infty}{v(t,a)da}$ in (\ref{newsystem}), the ordinary differential equation in (\ref{newsystem}) can be rewritten as the following age-structured model
\begin{equation*}
  \left\{
     \begin{array}{ll}
       \frac{\partial v(t,a)}{\partial t}+\frac{\partial v(t,a)}{\partial a} =-\tau\mu v(t,a),\\
       v(t,0)= \tau G(u(t,a),v(t,a)), \\
       v(0,a)=v_{0}\in L^{1}((0,+\infty),\mathbb{R}), \\
     \end{array}
   \right.
\end{equation*}
where
\begin{equation*}
\begin{array}{ccl}
  G(u(t,a),v(t,a)) & = & \Lambda\int_{0}^{+\infty}{v(t,a)da}-\frac{r}{K}\left(\int_{0}^{+\infty}{v(t,a)da}\right)^{2}
    -\frac{\alpha\int_{0}^{+\infty}{u(t,a)da}\int_{0}^{+\infty}{v(t,a)da}}{m\int_{0}^{+\infty}{u(t,a)da}+\int_{0}^{+\infty}{v(t,a)da}}.  \\
\end{array}
\end{equation*}
Let $w(t,a)=\left(
                           \begin{array}{c}
                             u(t,a) \\
                             v(t,a) \\
                           \end{array}
                         \right)
$, we can further obtain the equivalent system of model (\ref{system})
\begin{equation}\label{systempartialwta}
\left\{
  \begin{array}{l}
    \frac{\partial w(t,a)}{\partial t}+\frac{\partial w(t,a)}{\partial a} =-\tau Q w(t,a), \\
    w(t,0)=\tau B(w(t,a)), \\
    w(0,\cdot)=w_{0}=\left(
                       \begin{array}{c}
                         u_{0} \\
                        v_{0} \\
                       \end{array}
                     \right)\in L^{1}((0,+\infty),\mathbb{R}^{2}),
     \\
  \end{array}
\right.
\end{equation}
where
\begin{equation*}
  \begin{array}{ccc}
    Q=\left(
        \begin{array}{cc}
          \sigma & 0 \\
          0 & \mu \\
        \end{array}
      \right)
     & \mbox{and} & B(w(t,a))=\left(
                                   \begin{array}{c}
                                     \frac{\eta\alpha\int_{0}^{+\infty}{v(t,a)da} \int_{0}^{+\infty}{\beta(a)u(t,a)da}}
                                     {m\int_{0}^{+\infty}{u(t,a)da}+\int_{0}^{+\infty}{v(t,a)da}} \\
                                     G(u(t,a),v(t,a)) \\
                                   \end{array}
                                 \right).\\
  \end{array}
\end{equation*}
\noindent
Next we consider the following Banach space
\begin{equation*}
  X={\mathbb{R}}^{2} \times L^{1}{((0,+\infty),{\mathbb{R}}^{2})}
\end{equation*}
with $\left \|
\left(
  \begin{array}{c}
    \alpha \\
    \psi\\
  \end{array}
\right)
     \right \|
     =\left \|\alpha\right\|_{{\mathbb{R}}^{2}}+\left\|\psi\right\|_{L^{1}{((0,+\infty),{\mathbb{R}}^{2})}}$.
Define the linear operator $A_{\tau} : D(A_{\tau})\rightarrow X$ by
\begin{equation*}
  A_{\tau}\left(
     \begin{array}{c}
       0_{\mathbb{R}^{2}} \\
       \varphi \\
     \end{array}
   \right)
   =\left(
   \begin{array}{c}
     -\varphi(0) \\
     -\varphi'-\tau Q\varphi \\
   \end{array}
 \right)
\end{equation*}
with $D(A_{\tau})=\{0_{\mathbb{R}^{2}}\}\times W^{1,1}({(0,+\infty),{\mathbb{\mathbb{R}}}^{2}}) \subset X$, and the operator $H: \overline{D(A_{\tau})} \rightarrow X$ by
\begin{equation*}
  H\left(
  \left(
     \begin{array}{c}
       0_{\mathbb{R}^{2}} \\
       \varphi \\
     \end{array}
   \right)
   \right)
   =\left(
      \begin{array}{c}
        B(\varphi) \\
        0_{L^{1}} \\
      \end{array}
    \right).
\end{equation*}
The linear operator $A_{\tau}$ is non-densely defined because
\begin{equation*}
  X_{0}:=\overline{D(A_{\tau})}=\{0_{\mathbb{R}^{2}}\} \times L^{1}{((0,+\infty),{\mathbb{R}}^{2})}\neq X.
\end{equation*}
Set
\begin{equation*}
  x(t)=\left(
         \begin{array}{c}
           0_{\mathbb{R}^{2}} \\
           w(t,\cdot) \\
         \end{array}
       \right),
\end{equation*}
system (\ref{systempartialwta}) can be further rewritten as the following non-densely defined abstract Cauchy problem
\begin{equation}\label{nonddaCp}
  \left\{
    \begin{array}{l}
      \frac{dx(t)}{dt}  =  A_{\tau}x(t)+\tau H(x(t)), t\geq0, \\
      x(0)  =  \left(
                   \begin{array}{c}
                     0_{\mathbb{R}^{2}} \\
                     w_0 \\
                   \end{array}
                 \right)
                 \in \overline{D(A_{\tau})}. \\
    \end{array}
  \right.
\end{equation}
The global existence and uniqueness of solution of system (\ref{nonddaCp}) follow from the results of \cite{MagalRuan-ADE-2009} and \cite{Magal-EJDE-2001}.
\subsection{Equilibria and linearized equation}
\noindent

In this subsection, we will obtain the equilibria of system (\ref{nonddaCp}) and linearized equation of (\ref{nonddaCp}) around the positive equilibrium.
\subsubsection{Existence of equilibria}

\noindent

Suppose that $\overline{x}(a)=\left(
                      \begin{array}{c}
                        0_{\mathbb{R}^{2}} \\
                        \overline{w}(a) \\
                      \end{array}
                    \right)
                    \in X_0
$ is a steady state of system (\ref{nonddaCp}). Then
\begin{equation*}
  \left(
     \begin{array}{c}
       0_{\mathbb{R}^{2}} \\
       \overline{w}(a) \\
     \end{array}
   \right)\in D(A_{\tau}) \quad \mbox{and} \quad
   A_{\tau}\left(
      \begin{array}{c}
        0_{\mathbb{R}^{2}} \\
        \overline{w}(a) \\
      \end{array}
    \right)+\tau H\left(\left(
               \begin{array}{c}
                 0_{\mathbb{R}^{2}} \\
                 \overline{w}(a) \\
               \end{array}
             \right)\right)=0,
\end{equation*}
which is equivalent to
\begin{equation*}
\left\{
  \begin{array}{l}
    -\overline{w}(0)+\tau B(\overline{w}(a))=0, \\
    -\overline{w}^{'}(a)-\tau Q\overline{w}(a)=0. \\
  \end{array}
\right.
\end{equation*}
Moreover, we obtain
\begin{equation}\label{overlinewa}
  \left.
    \begin{array}{ccccc}
      \overline{w}(a)  = \left(
                              \begin{array}{c}
                                \overline{u}(a) \\
                                 \overline{v}(a) \\
                              \end{array}
                            \right)
        = \left(
               \begin{array}{c}
                 \tau\frac{\eta\alpha \overline{V}\int_{0}^{+\infty}{\beta(a)\overline{u}(a)da}}
                                     {m\int_{0}^{+\infty}{\overline{u}(a)da}+\overline{V}}
                                      e^{-\tau\sigma a}\\
                  \tau \left(
                  \Lambda\overline{V}-\frac{r}{K}\overline{V}^{2}
       -\frac{\alpha\overline{V}\int_{0}^{+\infty}{\overline{u}(a)da}}{m\int_{0}^{+\infty}{\overline{u}(a)da}+\overline{V}}
                  \right)e^{-\tau \mu a} \\
               \end{array}
             \right)
        \\
    \end{array}
  \right.
\end{equation}
with $\overline{V}=\int_{0}^{+\infty}{\overline{v}(a)da}$.

According to the first equation of (\ref{overlinewa}), we have
\begin{equation*}
\int_{0}^{+\infty}{\beta(a)\overline{u}(a)da}=\frac{\eta\alpha \overline{V}\int_{0}^{+\infty}{\beta(a)\overline{u}(a)da}}
                                     {m\int_{0}^{+\infty}{\overline{u}(a)da}+\overline{V}}
\quad\mbox{and}\quad  \int_{0}^{+\infty}{\overline{u}(a)da}=\frac{1}{\sigma}\int_{0}^{+\infty}{\beta(a)\overline{u}(a)da}.
\end{equation*}
On account of the second equation of (\ref{overlinewa}), we get
\begin{equation*}
  r\overline{V}-\frac{r}{K}\overline{V}^{2}
       -\frac{\alpha\overline{V}\int_{0}^{+\infty}{\overline{u}(a)da}}{m\int_{0}^{+\infty}{\overline{u}(a)da}+\overline{V}}=0.\\
\end{equation*}
Hence, we have the following lemma.
\begin{lemma}
System (\ref{nonddaCp}) has always the  equilibrium
  \begin{equation*}
\overline{x}_{0}(a)=\left(
                        \begin{array}{c}
                          0_{\mathbb{R}^{2}} \\
                          \left(
                            \begin{array}{c}
                              0_{L^{1}} \\
                              \tau\mu Ke^{-\tau\mu a} \\
                            \end{array}
                          \right)
                           \\
                        \end{array}
                      \right).
\end{equation*}
Furthermore, there exists a unique positive equilibrium of system (\ref{nonddaCp})
\begin{equation*}
\overline{x}_{\tau}=\left(
                  \begin{array}{c}
                    0_{\mathbb{R}^{2}} \\
                    \overline{w}_{\tau} \\
                  \end{array}
                \right)=\left(
  \begin{array}{c}
    0_{\mathbb{R}^{2}} \\
    \left(
      \begin{array}{c}
        \frac{K\sigma(\alpha\eta-1)[ m r\eta-(\alpha\eta-1)]}{ m^{2}r\eta}\tau e^{-\tau\sigma a} \\
        \frac{\mu K[ m r\eta-(\alpha\eta-1)]}{ m r\eta}\tau e^{-\tau\mu a} \\
      \end{array}
    \right)
  \end{array}
\right),
\end{equation*}
if and only if
\begin{equation*}
 m r\eta >\alpha\eta-1>0.
\end{equation*}
Correspondingly, there exists a unique positive equilibrium of system (\ref{system})
\begin{equation*}
  \left(
     \begin{array}{c}
      \overline{u}_{\tau}(a) \\
       \overline{V} \\
     \end{array}
   \right)=
   \left(
     \begin{array}{c}
       \frac{K\sigma(\alpha\eta-1)[ m r\eta-(\alpha\eta-1)]}{ m^{2}r\eta}\tau e^{-\tau\sigma a} \\
       \frac{K[ m r\eta-(\alpha\eta-1)]}{ m r\eta} \\
     \end{array}
   \right)
\end{equation*}
if and only if
\begin{equation*}
  m r\eta >\alpha\eta-1>0.
\end{equation*}
\end{lemma}
In the remainder of our paper we assume that $m r\eta >\alpha\eta-1>0$.

\subsubsection{Linearized equation}

\noindent

In order to get the linearized equation of (\ref{nonddaCp}) around the positive equilibrium $\overline{x}_{\tau}$, we first apply the following change of variable
\begin{equation*}
  y(t):= x(t)-\overline{x}_{\tau}.
\end{equation*}
Then, (\ref{nonddaCp}) becomes
\begin{equation}\label{system4}
  \left\{
    \begin{array}{cll}
      \frac{dy(t)}{dt} & = & A_{\tau}y(t)+\tau H(y(t)+\overline{x}_{\tau})-\tau H(\overline{x}_{\tau}), t\geq0, \\
      y(0) & = & \left(
                   \begin{array}{c}
                     0_{\mathbb{R}^{2}} \\
                     w_0-\overline{w}_{\tau} \\
                   \end{array}
                 \right)
                 =: y_0\in \overline{D(A_{\tau})}.
       \\
    \end{array}
  \right.
\end{equation}
Therefore the linearized equation (\ref{system4}) around the equilibrium $0$ is given by
\begin{equation}\label{systemlinear}
  \begin{array}{cc}
    \frac{dy(t)}{dt}=A_{\tau}y(t)+\tau DH(\overline{x}_{\tau})y(t) & \quad\mbox{for}\quad t\geq 0, y(t)\in X_{0}, \\
  \end{array}
\end{equation}
where
\begin{equation*}
    \begin{array}{cc}
      \tau DH(\overline{x}_{\tau})\left(
                        \begin{array}{c}
                          0_{\mathbb{R}^{2}} \\
                          \varphi \\
                        \end{array}
                      \right)=\left(
                                \begin{array}{c}
                                 \tau DB(\overline{w}_{\tau})(\varphi) \\
                                  0_{L^{1}} \\
                                \end{array}
                              \right)
       &\quad \mbox{for all}\quad\left(
                         \begin{array}{c}
                           0_{\mathbb{R}^{2}} \\
                           \varphi \\
                         \end{array}
                       \right)\in D(A_{\tau})\\
    \end{array}
\end{equation*}
with
\begin{equation*}
     \begin{array}{ccl}
       &&DB(\overline{w}_{\tau})(\varphi)=\\
       &&\left(
       \begin{array}{cc}
       -\frac{m\alpha\eta\overline{V}\int_{0}^{+\infty}{\beta(a)\overline{u}(a)da}}{\left(m\int_{0}^{+\infty}{\overline{u}(a)da}+\overline{V}\right)^{2}} & \frac{\alpha\eta\int_{0}^{+\infty}{\beta(a)\overline{u}(a)da}}{m\int_{0}^{+\infty}{\overline{u}(a)da}+\overline{V}}
       -\frac{\alpha\eta\overline{V}\int_{0}^{+\infty}{\beta(a)\overline{u}(a)da}}{\left(m\int_{0}^{+\infty}{\overline{u}(a)da}+\overline{V}\right)^{2}}\\
       -\frac{\alpha\overline{V}}{m\int_{0}^{+\infty}{\overline{u}(a)da}+\overline{V}}
       +\frac{\alpha m\overline{V}\int_{0}^{+\infty}{\overline{u}(a)da}}{\left(m\int_{0}^{+\infty}{\overline{u}(a)da}+\overline{V}\right)^{2}} & \Lambda-\frac{2r}{K}\overline{V}
       -\frac{\alpha\int_{0}^{+\infty}{\overline{u}(a)da}}{m\int_{0}^{+\infty}{\overline{u}(a)da}+\overline{V}}
       +\frac{\alpha\overline{V}\int_{0}^{+\infty}{\overline{u}(a)da}}
       {\left(m\int_{0}^{+\infty}{\overline{u}(a)da}+\overline{V}\right)^{2}} \\
       \end{array}
       \right)\\
       &&\times\int_{0}^{+\infty}{\varphi(a)da}
         +\left(
         \begin{array}{cc}
         \frac{\alpha\eta\overline{V}}{m\int_{0}^{+\infty}{\overline{u}(a)da}+\overline{V}}  & 0_{\mathbb{R}} \\
         0_{\mathbb{R}} & 0_{\mathbb{R}} \\
         \end{array}
         \right)
        \int_{0}^{+\infty}{\beta(a)\varphi(a)da}. \\
     \end{array}
\end{equation*}
Then we can rewrite system (\ref{system4}) as
\begin{equation}\label{fracdytdt}
    \begin{array}{cc}
      \frac{dy(t)}{dt}=B_{\tau}y(t)+\mathcal{H}(y(t)) &\quad \mbox{for}\quad t\geq0, \\
    \end{array}
\end{equation}
where
\begin{equation*}
  B_{\tau}:=A_{\tau}+\tau DH(\overline{x}_{\tau})
\end{equation*}
is a linear operator and
\begin{equation*}
  \mathcal{H}(y(t))=\tau H(y(t)+\overline{x}_{\tau})-\tau H(\overline{x}_{\tau})-\tau DH(\overline{x}_{\tau})y(t)
\end{equation*}
satisfying $\mathcal{H}(0)=0$ and $D\mathcal{H}(0)=0$.

\subsection{Characteristic equation}

\noindent

In this subsection, we will get the characteristic equation of (\ref{nonddaCp}) around the positive equilibrium $\overline{x}_{\tau}$.
Denote
\begin{equation*}
  \nu:=\min\{\sigma,\mu\}>0 \quad \mbox{and} \quad \Omega := \{\lambda \in \mathbb{C} : Re(\lambda)>-\nu\tau\}.
\end{equation*}
Following the results of \cite{LiuMagalRuan-ZAMP-2011}, we derive the following lemma.
\begin{lemma}
For $\lambda\in \Omega$, $\lambda\in \rho(A_{\tau})$ and
\begin{equation}\label{lambdaIAtau1}
  (\lambda I-A_{\tau})^{-1}\left(
                               \begin{array}{c}
                                 \delta \\
                                 \psi\\
                               \end{array}
                             \right)
                             =\left(
                                \begin{array}{c}
                                  0_{\mathbb{R}^{2}} \\
                                  \varphi \\
                                \end{array}
                              \right)
                              \Leftrightarrow
                              \varphi(a)=e^{-\int_{0}^{a}{(\lambda I+\tau Q)dl}}\delta+\int_{0}^{a}{e^{-\int_{s}^{a}{(\lambda I+\tau Q)dl}}\psi(s)}ds
\end{equation}
with $\left(
        \begin{array}{c}
          \delta \\
          \psi \\
        \end{array}
      \right)
      \in X
$ and $\left(
         \begin{array}{c}
           0_{\mathbb{R}^{2}} \\
           \varphi \\
         \end{array}
       \right)
       \in D(A_{\tau})
$. Furthermore, $A_{\tau}$ is a Hille-Yosida operator and
\begin{equation}\label{Hille-Yosida}
  \left\|(\lambda I-A_{\tau})^{-n}\right\|\leq\frac{1}{(Re(\lambda)+\nu\tau)^{n}},\forall\lambda\in\Omega,\forall n\geq 1.
\end{equation}
\end{lemma}
\noindent
Let $A_0$ be the part of $A_{\tau}$ in $\overline{D(A_{\tau})}$, namely, $A_0 := D(A_0)\subset X \rightarrow X$. For $\left(
                                                                                                    \begin{array}{c}
                                                                                                      0_{\mathbb{R}^{2}} \\
                                                                                                      \varphi \\
                                                                                                    \end{array}
                                                                                                  \right)
                                                                                                  \in D(A_0)
$, we get
\begin{equation*}
  A_0\left(
       \begin{array}{c}
         0_{\mathbb{R}^{2}} \\
         \varphi \\
       \end{array}
     \right)
     =\left(
        \begin{array}{c}
          0_{\mathbb{R}^{2}} \\
          \hat{A_0}(\varphi) \\
        \end{array}
      \right),
\end{equation*}
where $\hat{A_0}(\varphi)=-\varphi '-\tau Q\varphi$ with $D(\hat{A_0})=\{\varphi \in W^{1,1}((0,+\infty),{\mathbb{R}}^{2}): \varphi(0)=0\}$.
\noindent

Note that $\tau DH(\overline{v}_{\tau}):D(A_{\tau}) \subset X \rightarrow X$ is a compact bounded linear operator. From (\ref{Hille-Yosida}) we obtain
\begin{equation*}
  \left\| T_{A_0}(t) \right\| \leq e^{-\nu\tau t} \quad \mbox{for} \quad t \geq 0.
\end{equation*}
Thus, we have
\begin{equation*}
  \omega_{0,ess}(A_0)\leq\omega_0(A_{0})\leq -\nu\tau.
\end{equation*}
Combining with the perturbation results from \cite{DucrotLiuMagal-JMAA-2008}, we get
\begin{equation*}
  \omega_{0,ess}((A_{\tau}+\tau DH(\overline{x}_{\tau}))_{0})\leq-\nu\tau<0.
\end{equation*}
Consequently we derive the following proposition.
\begin{lemma}
The linear operator $B_{\tau}$ is a Hille-Yosida operator, and its part $(B_{\tau})_{0}$ in
$\overline{D(B_{\tau})}$ satisfies
\begin{equation*}
  \omega_{0,ess}((B_{\tau})_{0})<0.
\end{equation*}
\end{lemma}
\noindent
Let $\lambda\in \Omega$. Since $(\lambda I-A_{\tau})$ is invertible, and
\begin{equation}\label{invertible}
  \begin{array}{ccl}
    (\lambda I-B_{\tau})^{-1} & = & (\lambda I-(A_{\tau}+\tau DH(\overline{x}_{\tau})))^{-1} \\
     & = & (\lambda I-A_{\tau})^{-1}(I-\tau DH(\overline{x}_{\tau})(\lambda I-A_{\tau})^{-1})^{-1}, \\
  \end{array}
\end{equation}
it follows that $\lambda I-B_{\tau}$ is invertible if and only if $I-\tau DH(\overline{x}_{\tau})(\lambda I-A_{\tau})^{-1}$ is invertible. Set
\begin{equation*}
  (I-\tau DH(\overline{x}_{\tau})(\lambda I-A_{\tau})^{-1})\left(
     \begin{array}{c}
       \delta \\
       \varphi \\
     \end{array}
   \right)
   =\left(
      \begin{array}{c}
        \gamma \\
        \psi \\
      \end{array}
    \right).
\end{equation*}
It follows that
\begin{equation*}
  \left(
    \begin{array}{l}
      \delta \\
      \varphi \\
    \end{array}
  \right)
  -\tau DH(\overline{x}_{\tau})(\lambda I-A_{\tau})^{-1}\left(
                                                \begin{array}{c}
                                                  \delta \\
                                                  \varphi \\
                                                \end{array}
                                              \right)
                                              =\left(
                                                 \begin{array}{c}
                                                   \gamma \\
                                                   \psi \\
                                                 \end{array}
                                               \right).
\end{equation*}
Then we obtain
\begin{equation*}
 \left\{
    \begin{array}{l}
      \delta-\tau DB(\overline{w}_{\tau})\left(e^{-\int_{0}^{a}{(\lambda I+\tau Q)dl}}\delta+\int_{0}^{a}{e^{-\int_{s}^{a}{(\lambda I+\tau Q)dl}}\varphi(s)}ds\right)=\gamma, \\
      \varphi=\psi, \\
    \end{array}
  \right.
\end{equation*}
i.e.,
\begin{equation*}
  \left\{
    \begin{array}{l}
 \delta-\tau DB(\overline{w}_{\tau})\left(e^{  -\int_{0}^{a}{(\lambda I+\tau Q)dl}}\delta\right)=\gamma+\tau DB(\overline{w}_{\tau})\left(\int_{0}^{a}{e^{-\int_{s}^{a}{(\lambda I+\tau Q)dl}}\varphi(s)}ds\right), \\
      \varphi=\psi. \\
    \end{array}
  \right.
\end{equation*}
Taking the formula of $DB(\overline{w_{\tau}})$ into consideration, we obtain
\begin{equation*}
  \left\{
    \begin{array}{l}
      \Delta(\lambda)\delta=\gamma+K(\lambda,\psi), \\
      \varphi=\psi, \\
    \end{array}
  \right.
\end{equation*}
where
\begin{equation}\label{Deltalambda}
   \begin{array}{ccl}
    && \Delta(\lambda)  =  \\
     &&I-\left(
       \begin{array}{cc}
       -\frac{m\alpha\eta\overline{V}\int_{0}^{+\infty}{\beta(a)\overline{u}(a)da}}{\left(m\int_{0}^{+\infty}{\overline{u}(a)da}+\overline{V}\right)^{2}} & \frac{\alpha\eta\int_{0}^{+\infty}{\beta(a)\overline{u}(a)da}}{m\int_{0}^{+\infty}{\overline{u}(a)da}+\overline{V}}
       -\frac{\alpha\eta\overline{V}\int_{0}^{+\infty}{\beta(a)\overline{u}(a)da}}{\left(m\int_{0}^{+\infty}{\overline{u}(a)da}+\overline{V}\right)^{2}}\\
       -\frac{\alpha\overline{V}}{m\int_{0}^{+\infty}{\overline{u}(a)da}+\overline{V}}
       +\frac{\alpha m\overline{V}\int_{0}^{+\infty}{\overline{u}(a)da}}{\left(m\int_{0}^{+\infty}{\overline{u}(a)da}+\overline{V}\right)^{2}} & \Lambda-\frac{2r}{K}\overline{V}
       -\frac{\alpha\int_{0}^{+\infty}{\overline{u}(a)da}}{m\int_{0}^{+\infty}{\overline{u}(a)da}+\overline{V}}
       +\frac{\alpha\overline{V}\int_{0}^{+\infty}{\overline{u}(a)da}}
       {\left(m\int_{0}^{+\infty}{\overline{u}(a)da}+\overline{V}\right)^{2}} \\
       \end{array}
       \right)\\
       &&\times\tau\int_{0}^{+\infty}{e^{-\int_{0}^{a}{(\lambda I+\tau Q)dl}}da}
       -\left(
         \begin{array}{cc}
         \frac{\alpha\eta\overline{V}}{m\int_{0}^{+\infty}{\overline{u}(a)da}+\overline{V}}  & 0_{\mathbb{R}} \\
         0_{\mathbb{R}} & 0_{\mathbb{R}} \\
         \end{array}
         \right)\tau
                             \int_{0}^{+\infty}{\beta(a)e^{-\int_{0}^{a}{(\lambda I+\tau Q)dl}}}da \\
   \end{array}
\end{equation}
and
\begin{equation}\label{Klambdapsi}
  K(\lambda,\psi)=\tau DB(\overline{w}_{\tau})\left(\int_{0}^{a}{e^{-\int_{s}^{a}{(\lambda I+\tau Q)dl}}\psi(s)}ds\right).
\end{equation}
Whenever $\Delta(\lambda)$ is invertible, we have
\begin{equation}\label{xi}
  \delta=(\Delta(\lambda))^{-1}(\gamma+K(\lambda,\psi)).
\end{equation}
Combining the above discussion and the proof of Lemma 3.5 in \cite{WangLiu-JMAA-2012}, we obtain the following lemma.
\begin{lemma}
The following results hold
\begin{itemize}
  \item [(i)] $\sigma(B_{\tau})\cap\Omega=\sigma_{p}(B_{\tau})\cap\Omega=\{\lambda\in\Omega: \det(\Delta(\lambda))=0\}$;
  \item [(ii)] If $\lambda\in\rho(B_{\tau})\cap\Omega$, we have the following formula for resolvent
  \begin{equation}\label{lambdaIBtau}
    (\lambda I -B_{\tau})^{-1}\left(
                                \begin{array}{c}
                                  \delta \\
                                  \varphi \\
                                \end{array}
                              \right)
                              =\left(
                                 \begin{array}{c}
                                   0_{\mathbb{R}^{2}} \\
                                   \psi \\
                                 \end{array}
                               \right),
  \end{equation}
  where
  \begin{equation*}
    \psi(a)= e^{-\int_{0}^{a}{(\lambda I+\tau Q)dl}}(\Delta(\lambda))^{-1}\left[\gamma+K(\lambda,\varphi)\right]+\int_{0}^{a}
    {e^{-\int_{s}^{a}{(\lambda I+\tau Q)dl}}}\varphi(s)ds
  \end{equation*}
  with $\Delta(\lambda)$ and $K(\lambda,\varphi)$ defined in (\ref{Deltalambda}) and (\ref{Klambdapsi}).
\end{itemize}
\end{lemma}
\noindent
Under Assumption \ref{assumption1}, we have
\begin{equation}\label{intea}
  \int_{0}^{+\infty}{e^{-\int_{0}^{a}({\lambda I+\tau Q})dl}}da=
  \left(
    \begin{array}{cc}
      \frac{1}{\lambda+\sigma\tau} & 0 \\
      0 & \frac{1}{\lambda+\mu\tau} \\
    \end{array}
  \right)
\end{equation}
and
\begin{equation}\label{intbetaea}
  \int_{0}^{+\infty}{\beta(a)e^{-\int_{0}^{a}({\lambda I+\tau Q})dl}}da=
  \left(
    \begin{array}{cc}
      \frac{\beta^{*} e^{-(\lambda+\sigma\tau)}}{\lambda+\sigma\tau} & 0 \\
      0 & \frac{\beta^{*}e^{-(\lambda+\mu\tau)}}{\lambda+\mu\tau} \\
    \end{array}
  \right).
\end{equation}
It follows from (\ref{Deltalambda}), (\ref{intea}) and (\ref{intbetaea}) that the characteristic equation at the positive equilibrium $\overline{x}_{\tau}$ is
\begin{equation}\label{characteristicequation}
     \begin{array}{ccl}
       \det(\Delta(\lambda)) & = & \left|
  \begin{array}{cc}
                  1+\frac{\tau m \alpha\eta \overline{V}\xi}{(\lambda+\sigma\tau)\left(m\frac{\xi}{\sigma}+\overline{V}\right)^{2}}
                  -\frac{\tau\alpha \eta \overline{V}\sigma e^{-\lambda}}{(\lambda+\sigma\tau)\left(m\frac{\xi}{\sigma}+\overline{V}\right)}&

                  -\frac{\tau\left(
                  \frac{\alpha\eta\xi}{m\frac{\xi}{\sigma}+\overline{V}}
                  -\frac{\alpha\eta\overline{V}\xi}{\left(m\frac{\xi}{\sigma}+\overline{V}\right)^{2}}
                  \right)}{\lambda+\mu\tau}\\

                  -\frac{\tau\left(
                  -\frac{\alpha\overline{V}}{m\frac{\xi}{\sigma}+\overline{V}}
                  +\frac{m\alpha \overline{V}\xi}{\sigma\left(m\frac{\xi}{\sigma}+\overline{V}\right)^{2}}
                  \right)}{\lambda+\sigma\tau}&

                  1-\frac{\tau\left(
                  \Lambda-\frac{2r}{K}\overline{V}
                  -\frac{\alpha\xi}{\sigma\left(m\frac{\xi}{\sigma}+\overline{V}\right)}
                  +\frac{\alpha\overline{V}\xi}{\sigma\left(m\frac{\xi}{\sigma}+\overline{V}\right)^{2}}
                  \right)}{\lambda+\mu\tau}\\
                 \end{array}
  \right| \\
        & = & \frac{\lambda^{2}+\tau p_{1}\lambda+\tau^{2} p_{0}+(\tau q_{1}\lambda +\tau^{2}q_{0})e^{-\lambda}}{(\lambda+\sigma\tau)(\lambda+\mu\tau)} \\
        & \triangleq & \frac{\tilde{f}(\lambda)}{\tilde{g}(\lambda)}=0, \\
     \end{array}
\end{equation}
where
\begin{equation*}
    \begin{array}{lll}
      \xi & = & \int_{0}^{+\infty}{\beta(a)\overline{u}(a)da},\\
      \overline{V} & = & \int_{0}^{+\infty}{\overline{v}(a)da}, \\
      p_{1} & = & \sigma-r+\frac{2r}{K}\overline{V}+\frac{\alpha m\xi(\xi+\eta\sigma^{2}\overline{V})}{(m\xi+\sigma\overline{V})^{2}},\\
      p_{0} & = &\frac{2r\sigma\overline{V}}{K}
      -\frac{mr\alpha\eta\sigma^{2}\xi\overline{V}(K-2\overline{V})}{K(m\xi+\sigma\overline{V})^{2}}
      - \frac{\sigma[m\xi^{2}(mr-\alpha)+r\sigma\overline{V}(\sigma\overline{V}+2m\xi)]}{(m\xi+\sigma\overline{V})^{2}}
      +\frac{m\alpha^{2}\eta\sigma^{2}\xi^{2}\overline{V}}{(m\xi+\sigma\overline{V})^{3}},\\
     q_{1} & = & -\frac{\alpha\eta\sigma^{2}\overline{V}}{m\xi+\sigma\overline{V}},  \\
     q_{0} & = & \frac{\alpha\eta\sigma^{2}\overline{V}[r(K-2\overline{V})(m\xi+\sigma\overline{V})^{2}-K\alpha m \xi^{2}]}{K(m\xi+\sigma\overline{V})^{3}},  \\
      \tilde{f}(\lambda) & = & \lambda^{2}+\tau p_{1}\lambda+\tau^{2} p_{0}+(\tau q_{1}\lambda +\tau^{2}q_{0})e^{-\lambda},\\
      \tilde{g}(\lambda) & = & (\lambda+\sigma\tau)(\lambda+\mu\tau). \\
    \end{array}
\end{equation*}
Let
\begin{equation*}
  \lambda=\tau\zeta.
\end{equation*}
Then we get
\begin{equation}\label{characteristicequationg}
  \tilde{f}(\lambda)=\tilde{f}(\tau\zeta):=\tau^{2}g(\zeta)=\tau^{2}[\zeta^{2}+ p_{1}\zeta+p_{0}+(q_{1}\zeta + q_{0})e^{-\tau\zeta}].
\end{equation}
It is simple to prove that
\begin{equation*}
  \{\lambda\in\Omega:\det(\Delta(\lambda))=0\}=\{\lambda=\tau\zeta\in\Omega:g(\zeta)=0\}.
\end{equation*}

\section{Existence of Hopf bifurcation}

\noindent

In this section, we consider the parameter $\tau$ as a bifurcation parameter and study the existence of Hopf bifurcation by applying the Hopf bifurcation theory \cite{LiuMagalRuan-ZAMP-2011} to the Cauchy problem (\ref{nonddaCp}).
From (\ref{characteristicequationg}), we have
\begin{equation}\label{characteristicequation3}
 g(\zeta)=\zeta^{2}+ p_{1}\zeta+p_{0}+(q_{1}\zeta + q_{0})e^{-\tau\zeta},
\end{equation}
where
\begin{equation}\label{p1p2p3p4}
   \begin{array}{ll}
     p_{1}=\frac{m\alpha\eta^{2}(r+2\sigma)-\alpha^{2}\eta^{2}-m\sigma\eta+1}{m\alpha\eta^{2}},

     &p_{0}=\frac{\sigma[mr(2\alpha\eta-1)-2\alpha(\alpha\eta-1)]}{ m\alpha\eta},  \\

     q_{1}=-\sigma, &

     q_{0}=\frac{\sigma(-mr\alpha\eta^{2}+\alpha^{2}\eta^{2}-1)}{m\alpha\eta^{2}}. \\
   \end{array}
\end{equation}
Additionally,
if $ m r\eta >\alpha\eta-1>0$, then $ p_{0}+q_{0} = \frac{\sigma(\alpha\eta-1)[mr\eta-(\alpha\eta-1)]}{m\alpha\eta^{2}}>0$ and $\zeta =0$ is not a eigenvalue of (\ref%
{characteristicequation3}).

Let $\zeta=i\omega (\omega>0)$ be a purely imaginary root of $g(\zeta)=0$. Then we have
\begin{equation*}
  -\omega^{2}+ip_{1}\omega+p_{0}+(iq_{1}\omega+q_{0})e^{-i\omega\tau}=0.
\end{equation*}
Separating real and imaginary parts of the above equation gives rise to
\begin{equation}\label{realimaginary}
  \left\{
     \begin{array}{l}
       \omega^{2}-p_{0}=q_{1}\omega\sin(\omega\tau)+q_{0}\cos(\omega\tau), \\
       -p_{1}\omega=q_{1}\omega\cos(\omega\tau)-q_{0}\sin(\omega\tau). \\
     \end{array}
   \right.
\end{equation}
 Consequently, we obtain
\begin{equation*}
  (\omega^{2}-p_{0})^{2}+(-p_{1}\omega)^{2}=(q_{1}\omega)^2+q_{0}^{2},
\end{equation*}
that is,
\begin{equation}\label{omega34}
  \omega^{4}+(p_{1}^{2}-2p_{0}-q_{1}^{2})\omega^{2}+p_{0}^{2}-q_{0}^{2}=0.
\end{equation}
Set $\omega^{2}=\theta$, then (\ref{omega34}) becomes
\begin{equation}\label{siama2}
 \theta^{2}+(p_{1}^{2}-2p_{0}-q_{1}^{2})\theta+p_{0}^{2}-q_{0}^{2}=0.
\end{equation}
Let $\theta_{1}$ and $\theta_{2}$ denote two roots of (\ref{siama2}), then we find
\begin{equation}\label{sigma1sigma2}
  \theta_{1}+\theta_{2}=-(p_{1}^{2}-2p_{0}-q_{1}^{2}),\quad \theta_{1}\theta_{2}=p_{0}^{2}-q_{0}^{2}.
\end{equation}
 Consequently, it is apparent from (\ref{sigma1sigma2}) that when $p_{0}-q_{0}  =  \frac{\sigma[mr\eta(3\alpha\eta-1)-(3\alpha\eta+1)(\alpha\eta-1)]}{m\alpha\eta^{2}}<0(i.e.,mr\eta(3\alpha\eta-1)<(3\alpha\eta+1)(\alpha\eta-1))$, (\ref{siama2}) has only one positive real root $\theta_{0}$. Then (\ref{omega34}) has only one positive real root $\omega_{0}=\sqrt{\theta_{0}}$. According to (\ref{realimaginary}), we can yield that $g(\zeta)=0$ with $\tau=\tau_{k}$, $k=0,1,2,\cdots$ has a pair of purely imaginary roots $\pm i\omega_{0}$, where
\begin{equation*}
  \omega_{0}^{2}=\frac{-(p_{1}^{2}-2p_{0}-q_{1}^{2})+\sqrt{(p_{1}^{2}-2p_{0}-q_{1}^{2})^{2}-4(p_{0}^{2}-q_{0}^{2})}}{2}
\end{equation*}
and
\begin{equation}\label{tauk}
  \tau_{k}=\left\{
             \begin{array}{l}
               \frac{1}{\omega_{0}}\left(\arccos\frac{(q_{0}-p_{1}q_{1})\omega_{0}^{2}-p_{0}q_{0}}{q_{1}^{2}\omega_{0}^2+q_{0}^{2}}+2k\pi\right),
               \mbox{  if  } \frac{\omega_{0}(q_{1}\omega_{0}^{2}+p_{1}q_{0}-p_{0}q_{1})}{q_{1}^{2}\omega_{0}^2+q_{0}^{2}}\geq 0,\\
                \frac{1}{\omega_{0}}\left(2\pi-\arccos\frac{(q_{0}-p_{1}q_{1})\omega_{0}^{2}-p_{0}q_{0}}{q_{1}^{2}\omega_{0}^2+q_{0}^{2}}+2k\pi\right),
                \mbox{  if  } \frac{\omega_{0}(q_{1}\omega_{0}^{2}+p_{1}q_{0}-p_{0}q_{1})}{q_{1}^{2}\omega_{0}^2+q_{0}^{2}}< 0,\\
             \end{array}
           \right.
\end{equation}
for $k=0,1,2,\cdots.$
\begin{assumption}\label{assumption2}
  Assume that $m r\eta >\alpha\eta-1>0$ and $mr\eta(3\alpha\eta-1)<(3\alpha\eta+1)(\alpha\eta-1))$.
\end{assumption}
\begin{lemma}
Let Assumption \ref{assumption1} and \ref{assumption2} hold, then
\begin{equation*}
  \frac{\mbox{d}g(\zeta)}{\mbox{d}\zeta}\Big|_{\zeta=i\omega_{0}}\neq0.
\end{equation*}
Therefore, $\zeta=i\omega_{0}$ is a simple root of (\ref{characteristicequation3}).
\end{lemma}
\begin{proof}
On the basis of (\ref{characteristicequation3}), we have
\begin{equation*}
\frac{\mbox{d}g(\zeta)}{\mbox{d}\zeta}\Big|_{\zeta=i\omega_{0}}=\left\{2\zeta+p_{1}+[q_{1}-\tau (q_{1}\zeta+ q_{0})]e^{-\tau\zeta}\right\}\Big|_{\zeta=i\omega_{0}}
\end{equation*}
and
\begin{equation*}
  \left\{2\zeta+p_{1}+[q_{1}-\tau (q_{1}\zeta+ q_{0})]e^{-\tau\zeta}\right\}\frac{d\zeta(\tau)}{d\tau}=\zeta(q_{1}\zeta+q_{0})e^{-\tau\zeta}.
\end{equation*}
Suppose that $\frac{dg(\zeta)}{d\zeta}\Big|_{\zeta=i\omega_{0}}=0$, then
\begin{equation*}
  i\omega_{0}(iq_{1}\omega_{0}+q_{0})e^{-i\omega_{0}\tau}=0.
\end{equation*}
Separating real and imaginary in the above equation, we obtain
\begin{equation}
  \left\{
     \begin{array}{l}
       -q_{1}\omega_{0}^{2}\cos(\omega_{0}\tau)+q_{0}\omega_{0}\sin(\omega_{0}\tau)=0, \\
       q_{1}\omega_{0}^{2}\sin(\omega_{0}\tau)+q_{0}\omega_{0}\cos(\omega_{0}\tau)=0. \\
     \end{array}
   \right.
\end{equation}
That is,
\begin{equation*}
  (q_{1}\omega_{0}^{2})^{2}+(q_{0}\omega_{0})^{2}=0,
\end{equation*}
which implies
\begin{equation*}
  q_{1}\omega_{0}^{2}=q_{0}\omega_{0}=0.
\end{equation*}
Since $\omega_{0}>0$, we conclude that
\begin{equation*}
  q_{1}=q_{0}=0.
\end{equation*}
However, $q_{1}=-\sigma< 0$, which leads to a contradiction. Hence
\begin{equation*}
  \frac{\mbox{d}g(\zeta)}{\mbox{d}\zeta}\Big|_{\zeta=i\omega_{0}}\neq0.
\end{equation*}
This completes the proof.
\end{proof}
\begin{lemma}
  Let Assumption \ref{assumption1} and \ref{assumption2} hold. Denote the root $\zeta(\tau)=\alpha(\tau)+i\omega(\tau)$ of $g(\zeta)=0$ satisfying $\alpha(\tau_{k})=0$ and $\omega(\tau_{k})=\omega_{0}$, where $\tau_{k}$ is defined in (\ref{tauk}). Then
\begin{equation*}
  \alpha^{'}(\tau_{k})=\frac{\mbox{d}Re(\zeta)}{\mbox{d}\tau}\Big|_{\tau=\tau_{k}}>0.
\end{equation*}
\end{lemma}
\begin{proof}
For convenience, we study $\frac{d\tau}{d\zeta}$ instead of $\frac{d\zeta}{d\tau}$. From the expression of $g(\zeta)=0$, we have
\begin{equation*}
  \begin{array}{cll}
    \frac{\mbox{d}\tau}{\mbox{d}\zeta}\Big|_{\zeta=i\omega_{0}} & = & \frac{2\zeta+p_{1}+q_{1}e^{-\tau\zeta}-\tau(q_{1}\zeta +q_{0})e^{-\tau\zeta}}{\zeta(q_{1}\zeta +q_{0})e^{-\tau\zeta}}\Big|_{\zeta=i\omega_{0}} \\
     & = & \left(-\frac{2\zeta+p_{1}}{\zeta(\zeta^{2}+p_{1}\zeta+p_{0})}+\frac{q_{1}}{\zeta(q_{1}\zeta +q_{0})}-\frac{\tau}{\zeta}\right)\Big|_{\zeta=i\omega_{0}} \\
     & = &-\frac{i2\omega_{0}+p_{1}}{i\omega_{0}(-\omega_{0}^{2}+ip_{1}\omega_{0}+p_{0})}+\frac{q_{1}}{i\omega_{0}(iq_{1}\omega_{0} +q_{0})}-\frac{\tau}{i\omega_{0}}\\
     & = &\frac{1}{\omega_{0}}\frac{i2\omega_{0}+p_{1}}{p_{1}\omega_{0}-i(p_{0}-\omega_{0}^{2})}+\frac{1}{\omega_{0}}\frac{-q_{1}}{q_{1}\omega_{0}-iq_{0}}-\frac{\tau}{i\omega_{0}}\\
     & = & \frac{1}{\omega_{0}}\frac{(i2\omega_{0}+p_{1})[p_{1}\omega_{0}+i(p_{0}-\omega_{0}^{2})]}{(p_{1}\omega_{0})^{2}+(p_{0}-\omega_{0}^{2})^{2}}
     +\frac{1}{\omega_{0}}\frac{-q_{1}(q_{1}\omega_{0}+iq_{0})}{(q_{1}\omega_{0})^{2}+q_{0}^{2}}+\frac{i\tau}{\omega_{0}}.\\
  \end{array}
\end{equation*}
Therefore, we have
\begin{equation*}
     \begin{array}{ccl}
       $\mbox{Re}$\left(\frac{\mbox{d}\tau}{\mbox{d}\zeta}\Big|_{\zeta=i\omega_{0}} \right) & = & \frac{2\omega_{0}^{2}+p_{1}^{2}-2p_{0}}{(p_{1}\omega_{0})^{2}+(p_{0}-\omega_{0}^{2})^{2}}-\frac{q_{1}^{2}}{(q_{1}\omega_{0})^{2}+q_{0}^{2}} \\
        & = & \frac{2\omega_{0}^{2}+p_{1}^{2}-2p_{0}-q_{1}^{2}}{(q_{1}\omega_{0})^{2}+q_{0}^{2}}. \\
     \end{array}
\end{equation*}
Since
\begin{equation*}
  \omega_{0}^{2}=\frac{-(p_{1}^{2}-2p_{0}-q_{1}^{2})+\sqrt{(p_{1}^{2}-2p_{0}-q_{1}^{2})^{2}-4(p_{0}^{2}-q_{0}^{2})}}{2},
\end{equation*}
we can further obtain
\begin{equation*}
  \begin{array}{ccl}
    \mbox{sign}\left(\frac{\mbox{d}\rm{Re}(\zeta)}{\mbox{d}\tau}\Big|_{\tau=\tau_{k}}\right) & = & \mbox{sign}\left(\mbox{Re}\left(\frac{\mbox{d}\tau}{\mbox{d}\zeta}\Big|_{\zeta=i\omega_{0}}\right)\right) \\
     & = & \mbox{sign}\left(\frac{2\omega_{0}^{2}+p_{1}^{2}-2p_{0}-q_{1}^{2}}{(q_{1}\omega_{0})^{2}+q_{0}^{2}}\right)>0. \\
  \end{array}
\end{equation*}
\end{proof}
\noindent
Thus we conclude the following theorem.
\begin{theorem}\label{HopfBifurcation}
Let Assumption \ref{assumption1} and \ref{assumption2} hold. Then there exist $\tau_{k}>0, k=0,1,2,\cdots$($\tau_{k}$ is defined in (\ref{tauk})), such that when $\tau=\tau_{k}$, the predator-prey model (\ref{system}) undergoes a Hopf bifurcation at the equilibrium $(\overline{u}_{\tau_{k}}(a),\overline{V})$. In particular, a non-trivial periodic solution bifurcates from the equilibrium $(\overline{u}_{\tau_{k}}(a),\overline{V})$ when $\tau=\tau_{k}$.
\end{theorem}

\section{Numerical simulations}

\noindent

In this section, we perform some numerical simulations to illustrate the
results showed in Theorem \ref{HopfBifurcation}. We choose the parameter values:
$\Lambda=1.2, \mu=0.2, r=1, K=200, \alpha=2.35, m=1.66, \sigma=0.5, \eta=1$,
and the initial values
$u(0,\cdot)=30.3745e^{-a}$ and $V(0)=37.3494$.
The age-specific fertility function becomes
\begin{equation*}
  \beta(a):=\left\{
               \begin{array}{cl}
                 0.5e^{0.5\tau}, &\quad \mbox{if}\quad a\geq \tau, \\
                 0, &\quad \mbox{if}\quad a\in (0,\tau). \\
               \end{array}
             \right.
\end{equation*}
With the help of the Matlab, we can readily get $m r\eta-(\alpha\eta-1)\approx 0.3100$,
$\alpha\eta-1\approx 1.3500$, and $(3\alpha\eta+1)(\alpha\eta-1)-mr\eta(3\alpha\eta-1)\approx 0.8245$ which satisfy the conditions of Assumption \ref{assumption2}.
Calculating it further, we can easily obtain that $\omega_{0}\approx 0.1598$ and the first critical value $\tau_{0}\approx 1.9340$.
\begin{figure}[htbp]
\setlength{\belowcaptionskip}{1pt}
 \centering
\begin{minipage}[c]{0.4\textwidth}
    \centering
    \subfigure []
{\includegraphics[width=2.5in]{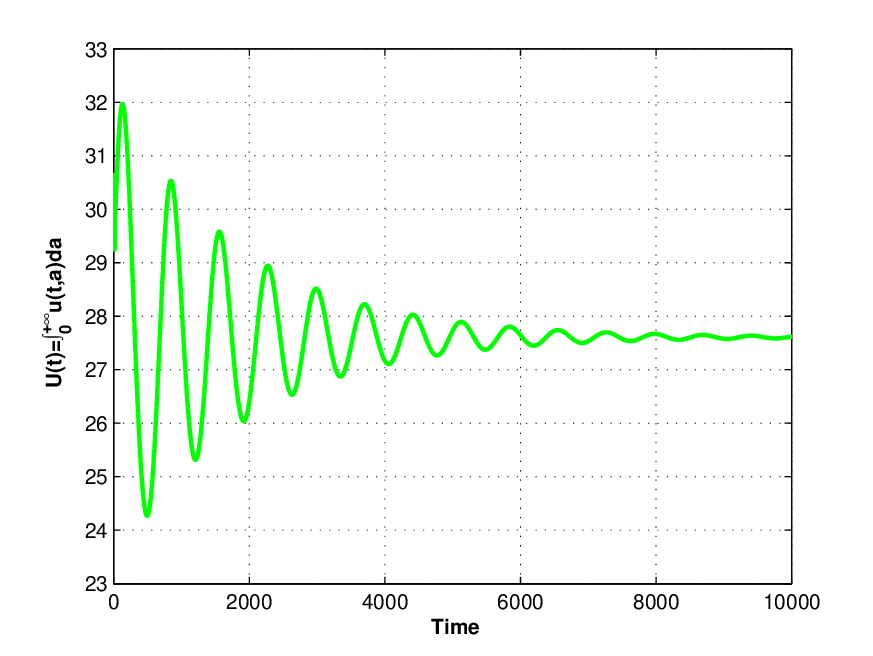}} \hfill
  \end{minipage}
\begin{minipage}[c]{0.4\textwidth}
    \centering
    \subfigure []
{\includegraphics[width=2.5in]{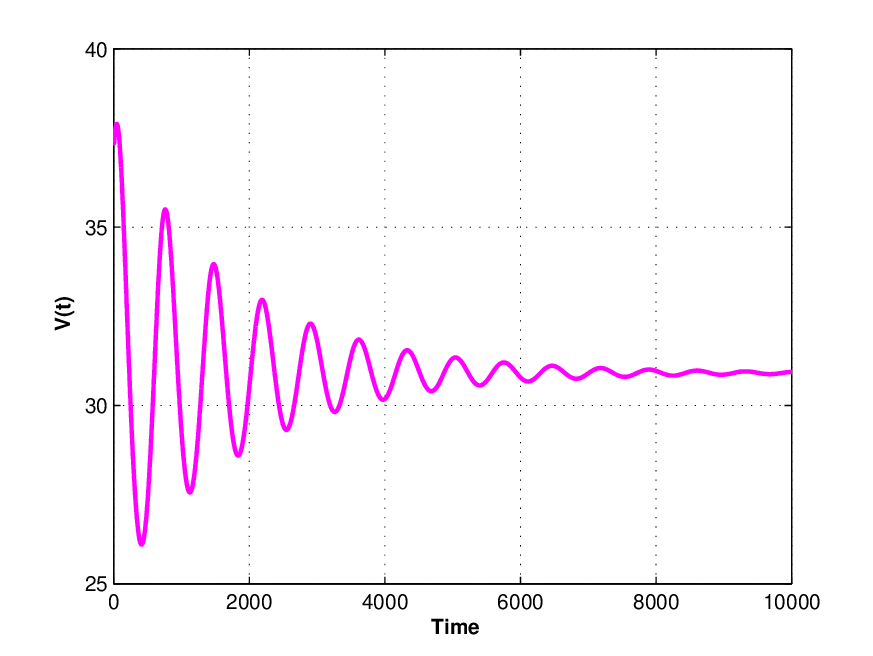}}
  \end{minipage}
\begin{minipage}[c]{0.4\textwidth}
    \centering
    \subfigure []
{\includegraphics[width=2.5in]{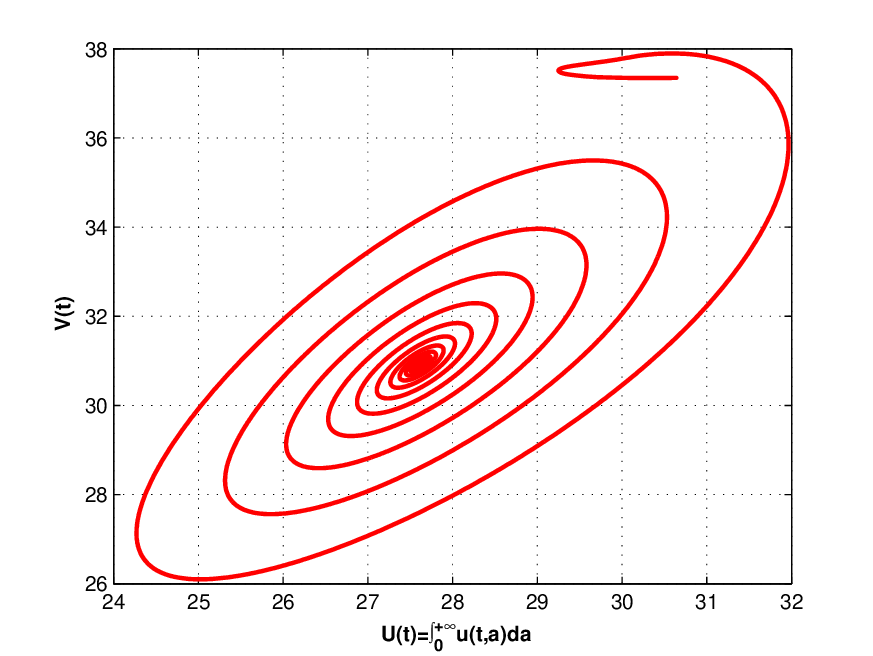}} \hfill
  \end{minipage}
\begin{minipage}[c]{0.4\textwidth}
    \centering
    \subfigure []
{\includegraphics[width=2.5in]{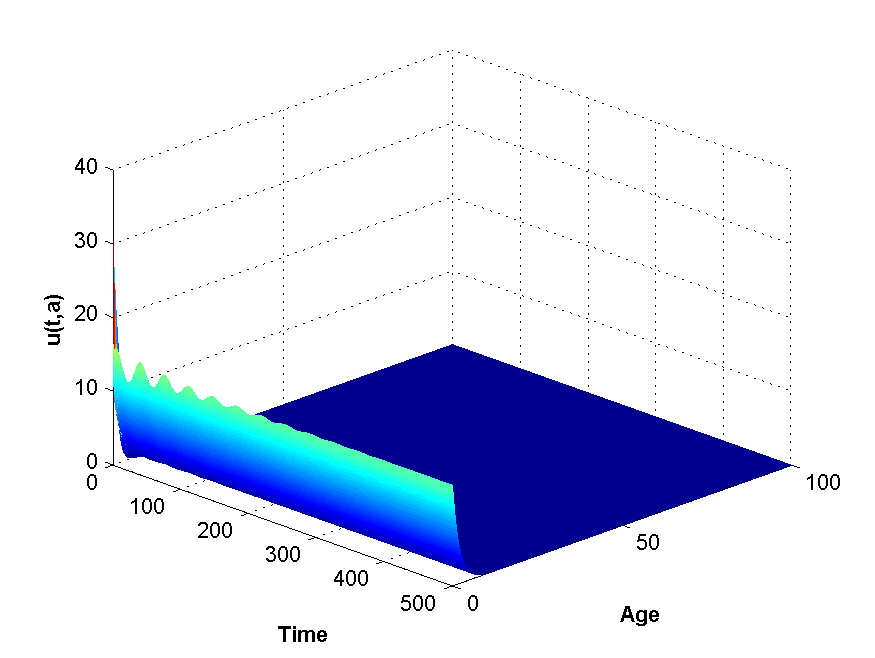}} \hfill
  \end{minipage}
\renewcommand{\figurename}{\footnotesize{Figure}}
\caption{\footnotesize{Numerical solutions of system (\ref{system}) when $\tau=0.9<\tau_{0}$: (a) solution behavior of the predator; (b) solution behavior of the prey; (c) phase trajectories for the system (\ref{system}); (d) distribution function of the predators $u(t,a)$. }}
\end{figure}

In Figure 1, we choose the bifurcation parameter $\tau=0.9<\tau_{0}$ and the positive equilibrium $(\overline{u}_{\tau=0.9}(a), \overline{V})=(13.6685e^{-0.45a}, 37.3494)$ is locally asymptotically stable. Figure 1(a) and Figure 1(b) demonstrate the solution behaviors of the predator and prey, respectively. Figure 1(c) reveals the phase diagram including $V(t)$ and $\int_{0}^{+\infty}{u(t,a)da}$
trajectories for the system (\ref{system}) and Figure 1(d) describes the change of the distribution function of the predators $u(t,a)$ as the time and age vary.
\begin{figure}[htbp]
\setlength{\belowcaptionskip}{1pt}
 \centering
\begin{minipage}[c]{0.4\textwidth}
    \centering
    \subfigure []
{\includegraphics[width=2.5in]{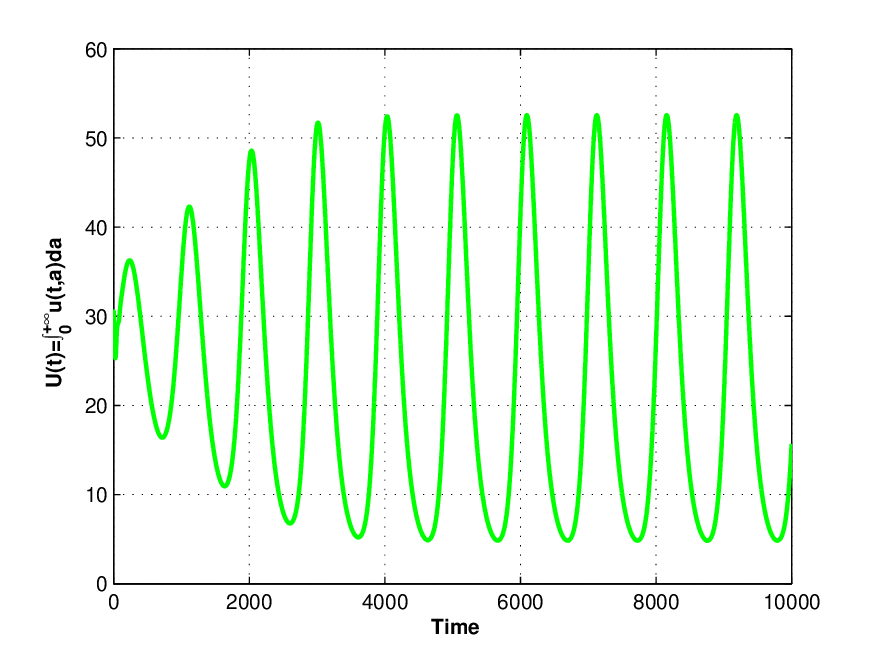}} \hfill
  \end{minipage}
\begin{minipage}[c]{0.4\textwidth}
    \centering
    \subfigure []
{\includegraphics[width=2.5in]{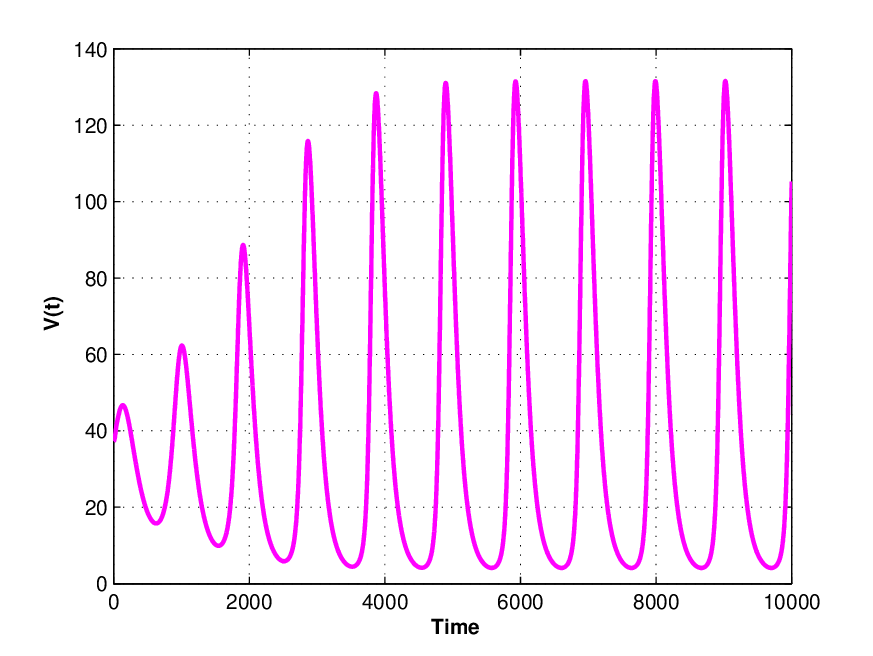}}
  \end{minipage}
\begin{minipage}[c]{0.4\textwidth}
    \centering
    \subfigure []
{\includegraphics[width=2.5in]{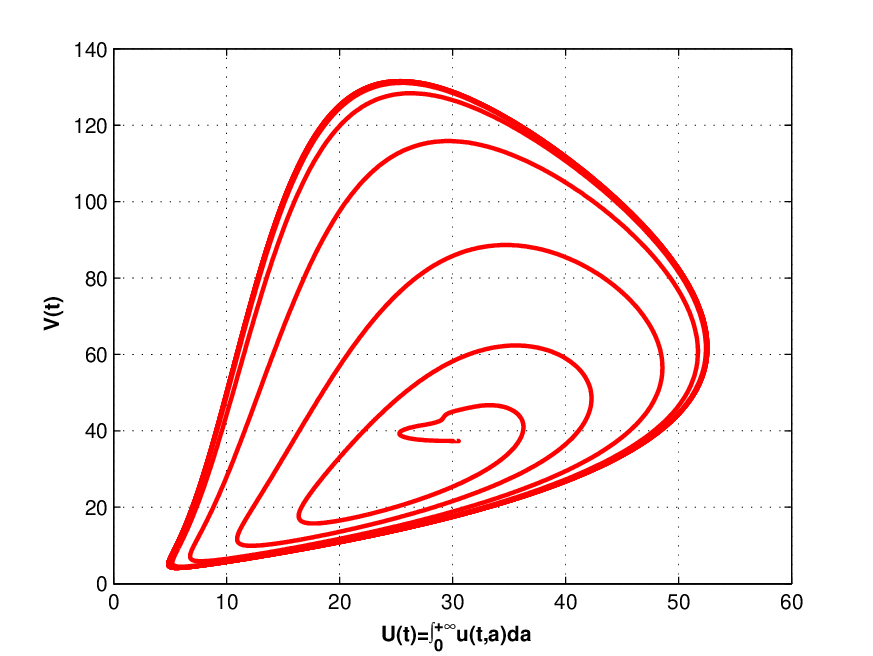}} \hfill
  \end{minipage}
\begin{minipage}[c]{0.4\textwidth}
    \centering
    \subfigure []
{\includegraphics[width=2.5in]{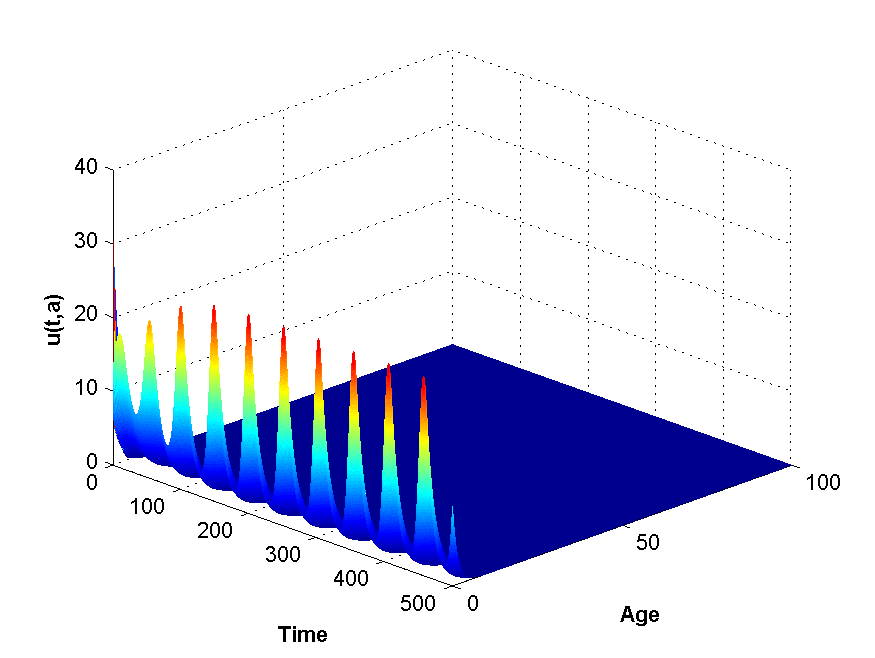}} \hfill
  \end{minipage}
\renewcommand{\figurename}{\footnotesize{Figure}}
\caption{\footnotesize{Numerical solutions of system (\ref{system}) when $\tau=2>\tau_{0}$: (a) periodic behavior of the predator; (b) periodic behavior of the prey; (c) phase portrait for the system (\ref{system}); (d) distribution function of the predators $u(t,a)$. }}
\end{figure}

By further continuously increasing $\tau$ to $2>\tau_{0}$, there appears a sustained periodic oscillation behavior of system (\ref{system}) around the positive equilibrium
$(\overline{u}_{\tau=2}(a), \overline{V})=(30.3745e^{-a}, 37.3494)$, meanwhile the conclusion of Theorem \ref{HopfBifurcation} is also numerically demonstrated (see Figure 2). In Figure 2(a) and Figure 2(b), the solution curves illustrate a sustained periodic oscillation behavior. As is shown in Figure 2(c), the oribt of $V(t)$ and $\int_{0}^{+\infty}{u(t,a)da}$
consistently approaches the stable limit cycles around this positive equilibrium. The
variation of $u(t,a)$ as time and predator-age vary at $\tau=2>\tau_{0}$ is demonstrated in Figure 2(d).

\section{Sensitivity analysis}

\noindent

In this section, we illustrate the influence of several important parameters on the dynamics of the predator population and prey population through graphical approach. The parameter values and the initial values are the same as Section 4.
\begin{figure}[htbp]
\setlength{\belowcaptionskip}{1pt}
 \centering
\begin{minipage}[c]{0.4\textwidth}
    \centering
    \subfigure []
{\includegraphics[width=2.5in]{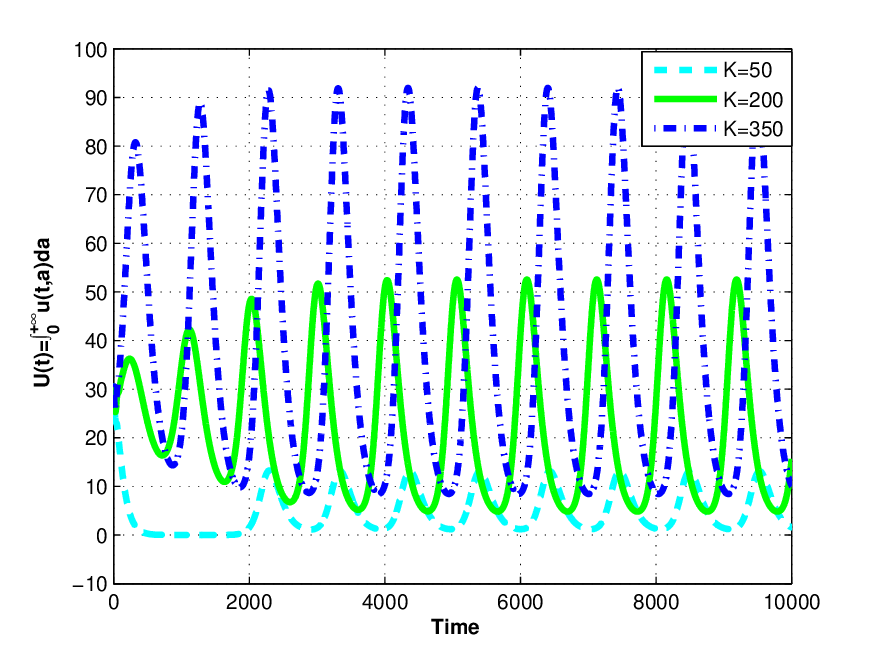}} \hfill
  \end{minipage}
\begin{minipage}[c]{0.4\textwidth}
    \centering
    \subfigure []
{\includegraphics[width=2.5in]{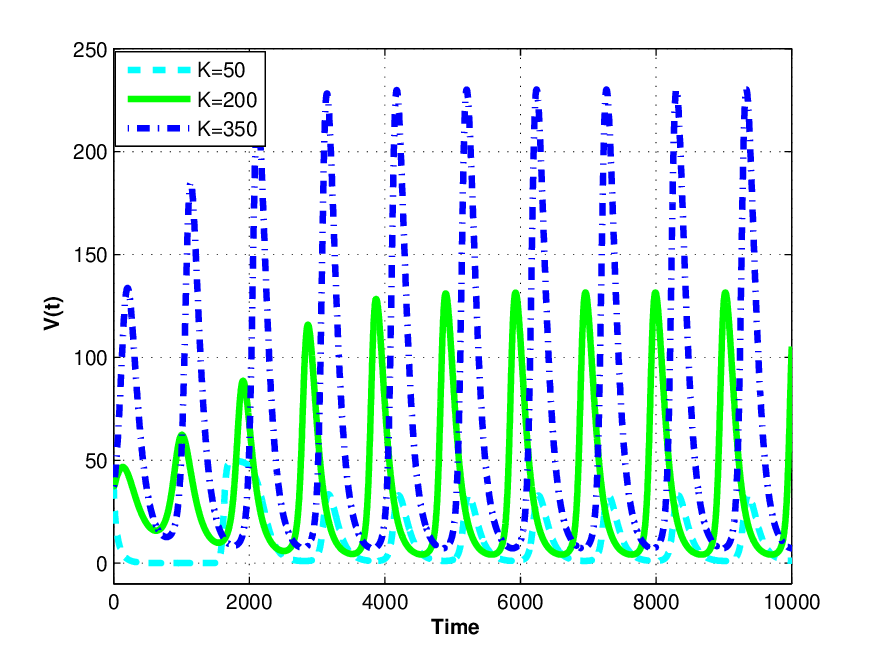}}
  \end{minipage}
  \begin{minipage}[c]{0.4\textwidth}
    \centering
    \subfigure []
{\includegraphics[width=2.5in]{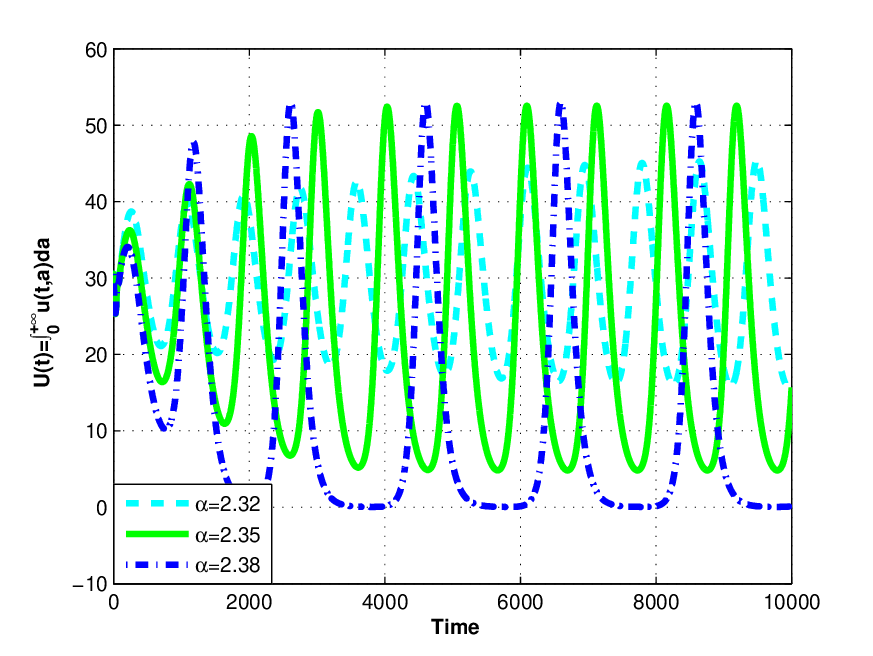}} \hfill
  \end{minipage}
\begin{minipage}[c]{0.4\textwidth}
    \centering
    \subfigure []
{\includegraphics[width=2.5in]{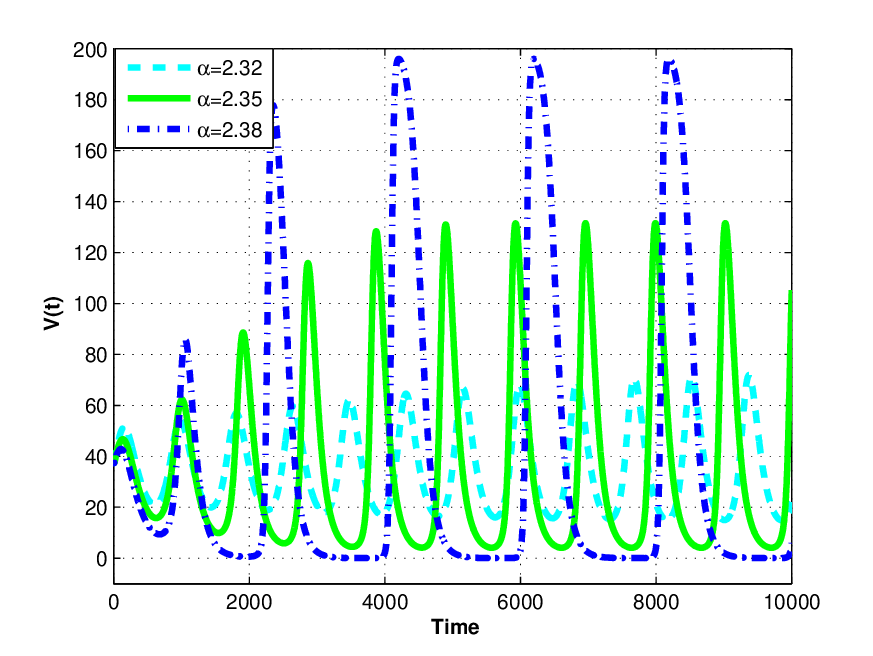}}
  \end{minipage}
\begin{minipage}[c]{0.4\textwidth}
    \centering
    \subfigure []
{\includegraphics[width=2.5in]{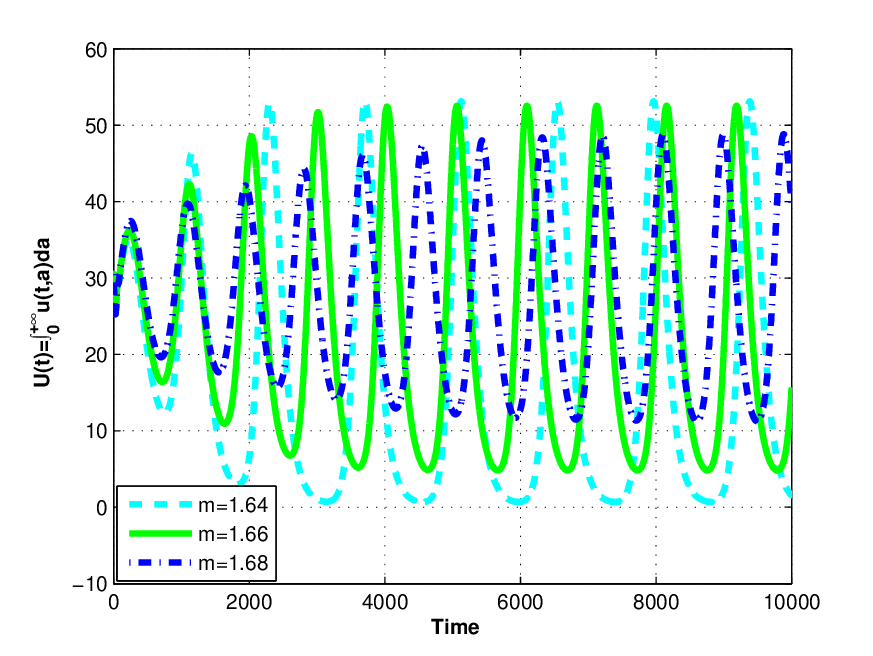}} \hfill
  \end{minipage}
\begin{minipage}[c]{0.4\textwidth}
    \centering
    \subfigure []
{\includegraphics[width=2.5in]{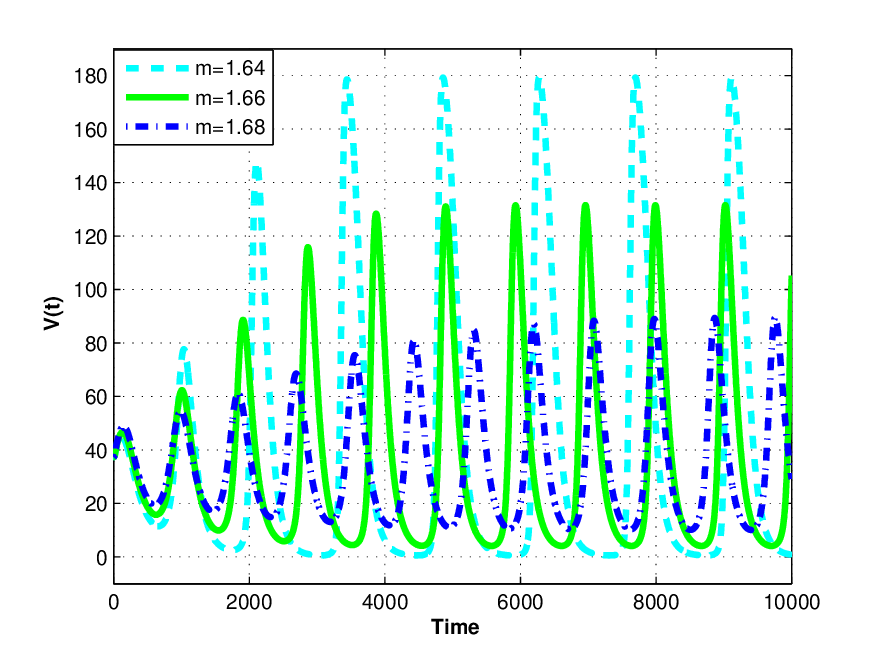}}
  \end{minipage}
\renewcommand{\figurename}{\footnotesize{Figure}}
\caption{\footnotesize{The effect of the parameters $K$, $\alpha$ and $m$ on the dynamics of the predator and prey populations. }}
\end{figure}

Figure 3(a) and Figure 3(b) show that the carrying capacity of prey $K$ has a greater impact on both prey and predator.
Figure 3(c) and Figure 3(d) illustrate the difference of the dynamics of predator and prey populations in terms of the different capturing rate $\alpha$. When the capturing rate increases gradually, the amplitude of the periodic oscillation phenomena of the two populations become bigger and bigger. We can readily find that when the capturing rate exceeds a certain value $(\alpha=2.35)$, the effect of the catching rate on predator population is less than the prey population.
Compared with the capturing rate $\alpha$ (see Figure 3(c) and Figure 3(d)), the effect of the half capturing saturation constant $m$ (see Figure 3(e) and Figure 3(f)) on the dynamics of system (\ref{system}) is just the opposite.
As is shown in Figure 3(e) and Figure 3(f), the amplitude of the periodic oscillation behaviors gradually decrease with the increase of
the half capturing saturation constant $m$. Obviously, in comparison with Figure 3(e), the change in Figure 3(f) is more evident.
Comparing Figure 4(a), Figure 4(b), Figure 3(c) and Figure 3(d), we can readily observe that the effect of the conversion rate $\eta$ (see Figure 4(a) and Figure 4(b)) on the dynamic behaviors of system (\ref{system}) is consistent with the effect of the capturing rate $\alpha$. The amplitude of the solution curves of system (\ref{system}) demonstrate an increase tendency with the increase of the conversion rate $\eta$.
\begin{figure}[htbp]
\setlength{\belowcaptionskip}{1pt}
 \centering
 \begin{minipage}[c]{0.4\textwidth}
    \centering
    \subfigure []
{\includegraphics[width=2.5in]{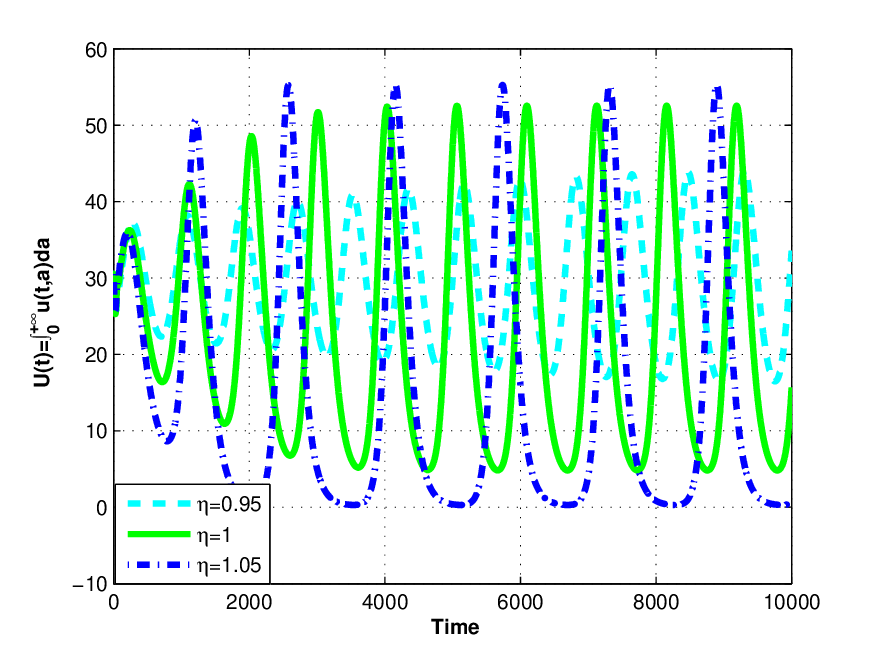}} \hfill
  \end{minipage}
\begin{minipage}[c]{0.4\textwidth}
    \centering
    \subfigure []
{\includegraphics[width=2.5in]{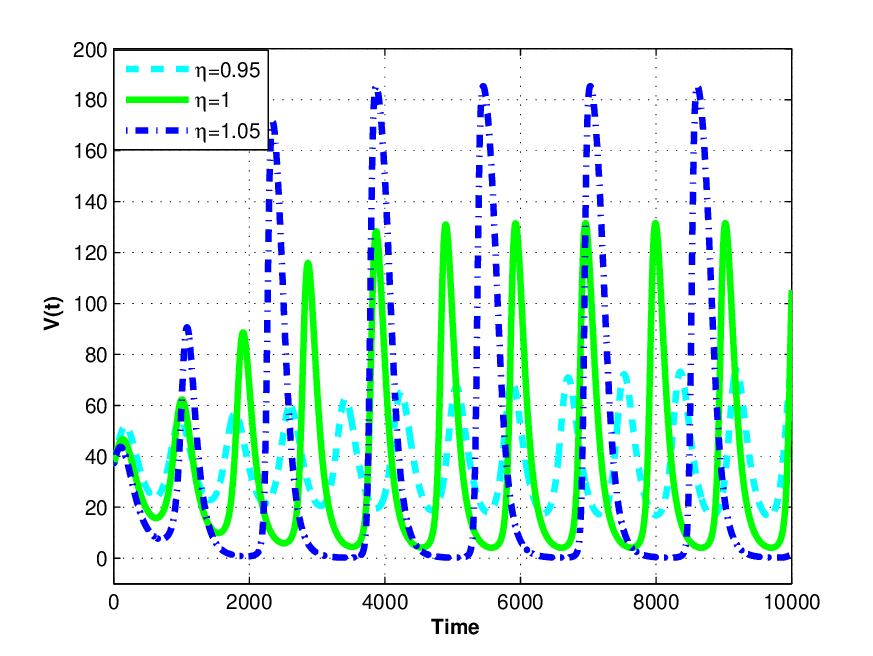}}
  \end{minipage}
\begin{minipage}[c]{0.4\textwidth}
    \centering
    \subfigure []
{\includegraphics[width=2.5in]{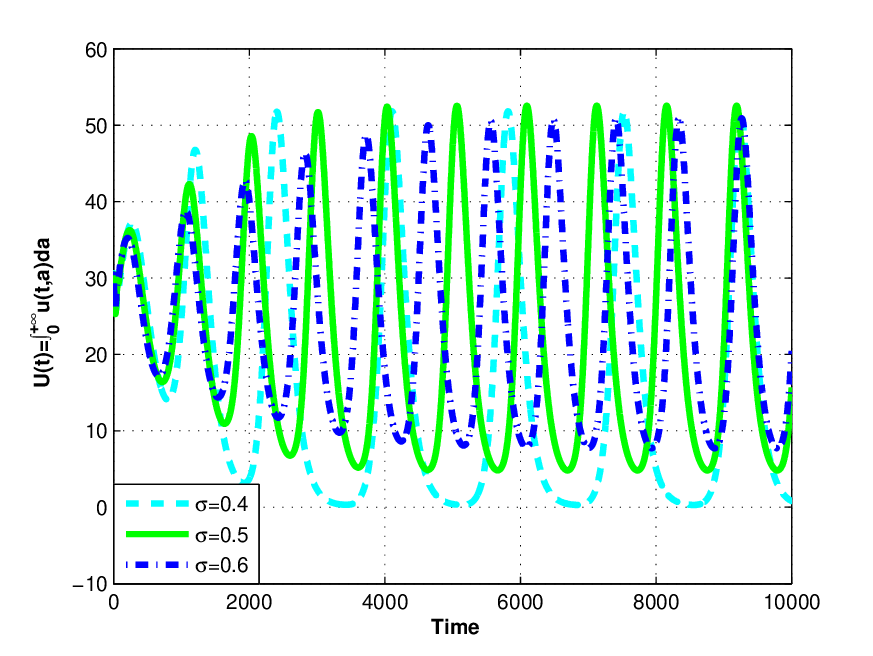}} \hfill
  \end{minipage}
\begin{minipage}[c]{0.4\textwidth}
    \centering
    \subfigure []
{\includegraphics[width=2.5in]{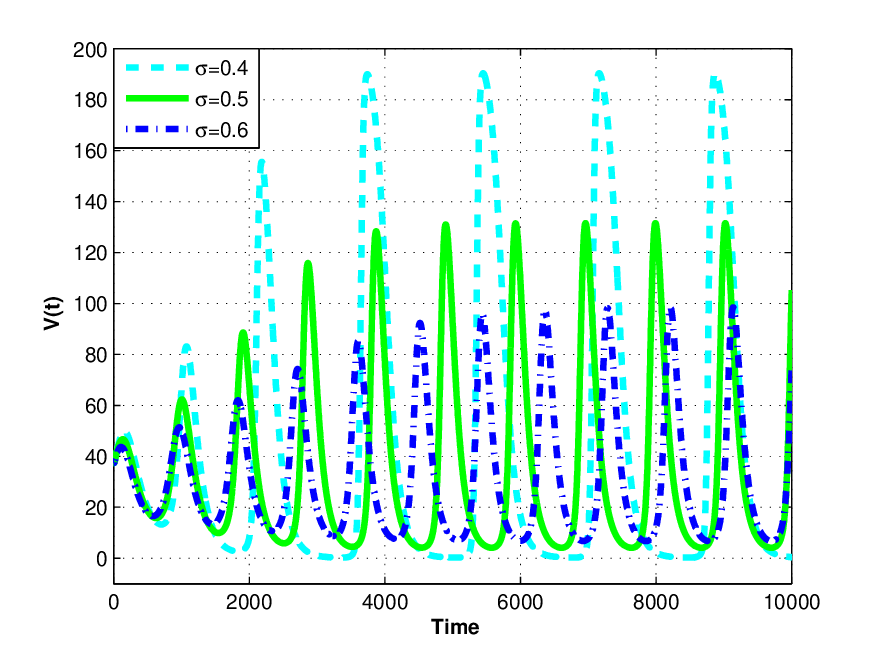}}
  \end{minipage}
\renewcommand{\figurename}{\footnotesize{Figure}}
\caption{\footnotesize{The effect of the parameters $\eta$ and $\sigma$ on the dynamics of the predator and prey populations.}}
\end{figure}
The effect of predator mortality rate $\sigma$ on the dynamics of the system (\ref{system}) is also obvious.
From the Figure 4(c) and Figure 4(d), we can clearly see that as the predator mortality rate $\sigma$ increases gradually, the amplitude of the periodic oscillation behaviors become smaller and smaller. In contrast with the predator population (see Figure 4(c)), the predator mortality rate has a greater impact on the dynamic behaviors of the prey population (see Figure 4(d)).

\section{Conclusions}

\noindent

In our model (\ref{system}), we introduce a predator-prey model with predator-age structure that involves Michaelis-Menten type ratio-dependent functional response. Our results demonstrate that when the bifurcation parameter $\tau$ passes through the critical value $\tau_{k} (k=0,1,2,\cdots)$, the Hopf bifurcation occurs around the positive equilibrium of the system (\ref{system}). Biologically the bifurcation parameter $\tau$ might be taken as a measure of a biological maturation period. Based on the theoretical analysis and numerical simulations, we conclude that the stability of the unique positive equilibrium of system (%
\ref{system}) is unaffected when the biological maturation period $\tau$ is small enough. However, when the maturation period $\tau$
crosses critical value $\tau _{k} (k=0,1,2,\cdots)$, the sustained periodic
oscillation phenomena appear around the positive equilibrium. On the basis of the sensitivity analysis, graphical method illustrates that the effect of parameters $\alpha$, $m$, $\eta$ and $\sigma$ on the dynamics of the prey population is more obvious than the predator population. However, the parameter $K$ has a greater impact on both prey and predator.

\end{document}